\pdfoutput=1
\documentclass[a4paper,11pt]{amsart}

\usepackage[margin=1in,includehead,includefoot]{geometry}
\usepackage[utf8]{inputenc}
\usepackage[T1]{fontenc}
\usepackage{lmodern,microtype}
\usepackage{amsmath,amsthm,amssymb}
\usepackage{enumitem}
\usepackage{tikz}
\usepackage{hyperref,bookmark}

\hypersetup{colorlinks=true,linkcolor=black,citecolor=black,filecolor=black,urlcolor=black}

\makeatletter
\let\citationorig\citation
\def\citation#1{\citationorig{#1}\@for\@tempa:=#1\do{\@ifundefined{cit@\@tempa}{\global\@namedef{cit@\@tempa}{}}{}}}
\let\bibitemorig\bibitem
\def\bibitem#1{\@ifundefined{cit@#1}{\typeout{LaTeX Warning: Unused bibitem `#1'}}{}\bibitemorig{#1}}
\let\old@setaddresses\@setaddresses
\def\@setaddresses{\bigskip{\parindent 0pt\let\scshape\relax\let\ttfamily\relax\old@setaddresses}}

\makeatother

\newtheorem{theorem}{Theorem}
\newtheorem{lemma}{Lemma}
\newtheorem{proposition}{Proposition}
\newtheorem{claim}{Claim}
\theoremstyle{remark}
\newtheorem*{block-prop}{Block Restriction Property}
\newtheorem*{interval-prop}{Interval Property for Tails}

\renewenvironment{enumerate}{\begin{enumorig}[label=\textup{(\arabic*)}, noitemsep, topsep=1mm, labelindent=.5em, leftmargin=*]}{\end{enumorig}}

\newcommand{\bfk}{\mathbf{k}}
\newcommand{\bfn}{\mathbf{n}}
\newcommand{\cgB}{\mathcal{B}}
\newcommand{\cgF}{\mathcal{F}}
\newcommand{\cgR}{\mathcal{R}}

\DeclareMathOperator{\Inc}{Inc}
\DeclareMathOperator{\Max}{Max}

\let\leq\leqslant
\let\geq\geqslant

\allowdisplaybreaks

\hypersetup{
  pdftitle={Dimension and cut vertices: an application of Ramsey theory},
  pdfauthor={William T. Trotter, Bartosz Walczak, Ruidong Wang},
}

\title[Dimension and cut vertices: an application of Ramsey theory]{Dimension and cut vertices:\\an application of Ramsey theory}

\author{William~T. Trotter\and Bartosz Walczak\and Ruidong Wang}

\address[William~T. Trotter]{School of Mathematics, Georgia Institute of Technology, Atlanta, GA, USA}
\email{\href{mailto:trotter@math.gatech.edu}{trotter@math.gatech.edu}}

\address[Bartosz Walczak]{Department of Theoretical Computer Science, Faculty of Mathematics and Computer Science, Jagiellonian University, Kraków, Poland}
\email{\href{mailto:walczak@tcs.uj.edu.pl}{walczak@tcs.uj.edu.pl}}

\address[Ruidong Wang]{Blizzard Entertainment, Irvine, CA, USA}
\email{\href{mailto:ruwang@blizzard.com}{ruwang@blizzard.com}}

\dedicatory{This paper is dedicated to Ronald~L. Graham in celebration of his 80th birthday and in appreciation\\for the distinguished leadership he has provided for combinatorial mathematics.}

\thanks{This paper appeared as Chapter 11 in: \href{https://doi.org/10.1017/9781316650295.012}{Steve Butler, Joshua Cooper, and Glenn Hurlbert (eds.), \emph{Connections in Discrete Mathematics:\ A Celebration of the Work of Ron Graham}, pp.~187--199, Cambridge University Press, Cambridge, 2018}.}
\thanks{Bartosz Walczak was partially supported by National Science Center of Poland grant 2011/03/N/ST6/03111.}

\begin{document}

\begin{abstract}
Motivated by quite recent research involving the relationship between the dimension of a poset and graph-theoretic properties of its cover graph, we show that for every $d\geq 1$, if $P$ is a poset and the dimension of a subposet $B$ of $P$ is at most $d$ whenever the cover graph of $B$ is a block of the cover graph of $P$, then the dimension of $P$ is at most $d+2$.
We also construct examples which show that this inequality is best possible.
We consider the proof of the upper bound to be fairly elegant and relatively compact.
However, we know of no simple proof for the lower bound, and our argument requires a powerful tool known as the Product Ramsey Theorem.
As a consequence, our constructions involve posets of enormous size.
\end{abstract}

\maketitle

\section{Introduction}

We assume that the reader is familiar with basic notation and terminology for partially ordered sets (here we use the short term \emph{posets}), including chains and antichains, minimal and maximal elements, linear extensions, order diagrams, and cover graphs.
Extensive background information on the combinatorics of posets can be found in \cite{Tro-book,Tro95}.

We will also assume that the reader is familiar with basic concepts of graph theory, including the following terms: connected and disconnected graphs, components, cut vertices, and $k$-connected graphs for an integer $k\geq 2$.
Recall that when $G$ is a connected graph, a connected induced subgraph $H$ of $G$ is called a \emph{block} of $G$ when $H$ is $2$-connected and there is no subgraph $H'$ of $G$ which contains $H$ as a proper subgraph and is also $2$-connected.

Here are the analogous concepts for posets.
A poset $P$ is said to be \emph{connected} if its cover graph is connected.
A subposet $B$ of $P$ is said to be \emph{convex} if $y\in B$ whenever $x,z\in B$ and $x<y<z$ in $P$.
Note that when $B$ is a convex subposet of $P$, the cover graph of $B$ is an induced subgraph of the cover graph of $P$.
A convex subposet $B$ of $P$ is called a \emph{component} of $P$ when the cover graph of $B$ is a component of the cover graph of $P$.
A convex subposet $B$ of $P$ is called a \emph{block} of $P$ when the cover graph of $B$ is a block in the cover graph of $P$.

Motivated by questions raised in recent papers exploring connections between the dimension of a poset $P$ and graph-theoretic properties of the cover graph of $P$, our main theorem will be the following result.

\begin{theorem}
\label{thm:main}
For every\/ $d\geq 1$, if\/ $P$ is a poset and every block in\/ $P$ has dimension at most\/ $d$, then the dimension of\/ $P$ is at most\/ $d+2$.
Furthermore, this inequality is best possible.
\end{theorem}

The remainder of this paper is organized as follows.
In the next section, we present a brief discussion of background material which serves to motivate this line of research and puts our theorem in historical perspective.
Section~\ref{sec:dimension} includes a compact summary of essential material from dimension theory.
Section~\ref{sec:upper-bound} contains the proof of the upper bound in our main theorem, and in Section~\ref{sec:best-possible}, we give a construction which shows that our upper bound is best possible.
This construction uses the Product Ramsey Theorem and produces posets of enormous size.
We close in Section~\ref{sec:closing} with some brief remarks about challenges that remain.

\section{Background motivation}

A family $\cgF=\{L_1,L_2,\ldots,L_d\}$ of linear extensions of a poset $P$ is called a \emph{realizer} of $P$ when $x\leq y$ in $P$ if and only if $x\leq y$ in $L_i$ for each $i=1,2,\ldots,d$.
The \emph{dimension} of $P$, denoted by $\dim(P)$, is the least positive integer $d$ for which $P$ has a realizer of size $d$.
For simplifying the details of arguments to follow, we consider families of linear extensions with repetition allowed.
So if $\dim(P)=d'$, then $P$ has a realizer of size $d$ for every $d\geq d'$.
For an integer $d\geq 2$, a poset $P$ is said to be \emph{$d$-irreducible} if $\dim(P)=d$ and $\dim(B)<d$ for every proper subposet $B$ of $P$.

As is well known, the dimension of a poset $P$ is just the maximum of the dimension of the components of $P$ \emph{except} when $P$ is the disjoint sum of two or more chains.
In the latter case, $\dim(P)=2$ while all components of $P$ have dimension $1$.
Accordingly, when $d\geq 3$, a poset $P$ with $\dim(P)=d$ has a component $Q$ with $\dim(Q)=d$.

It is easy to see that if the chromatic number of a connected graph $G$ is $r$ and $r\geq 2$, then there is a block $H$ of $G$ so that the chromatic number of $H$ is also $r$.
The analogous statement for posets is not true.
We show in Figure~\ref{fig:cut-vertices} representatives of two infinite families of posets.
When $n\geq 2$, the poset $P_n$ shown on the left side is $3$-irreducible (see \cite{TM76} or \cite{Kel77} for the full list of $3$-irreducible posets).
For each $m\geq 3$, the poset $Q_m$ shown on the right is $(m+1)$-irreducible.
This second example is part of an exercise given on page $20$ in \cite{Tro-book}.
Together, these examples show that for every $d\geq 2$, there are posets of dimension $d+1$ every block of which has dimension at most $d$.

\begin{figure}
\begin{center}
\begin{minipage}[b]{6.5cm}
\centering
\begin{tikzpicture}
  \tikzstyle{every node}=[circle,minimum size=4pt,inner sep=0pt,draw]
  \tikzstyle{every label}=[rectangle,label distance=3pt,draw=none]
  \node (a) at (0.25,0) {};
  \node[label=left:$1$] (b1) at (-2,1) {};
  \node[label=left:$2$] (b2) at (-1,1) {};
  \node[label=left:$3$] (b3) at (0,1) {};
  \node (b4) at (1.5,1) {};
  \node[rectangle,draw=none,yshift=-0.27pt] at (0.75,1) {$\cdots$};
  \node[label=left:$\vphantom{1}n$] (b5) at (2.5,1) {};
  \node (e) at (3.5,1.2) {};
  \node[draw=none] (b4') at (0.75,1.5) {};
  \node[draw=none] (c4') at (1.25,1.5) {};
  \node (c1) at (-2.5,2) {};
  \node (c2) at (-1.5,2) {};
  \node (c3) at (-0.5,2) {};
  \node (c4) at (0.5,2) {};
  \node[rectangle,draw=none,yshift=-0.27pt] at (1.25,2) {$\cdots$};
  \node (c5) at (2,2) {};
  \node (d) at (0.25,3) {};
  \path (a) edge (b1) edge (b2) edge (b3) edge (b4) edge[bend right=10] (e);
  \path (d) edge (c2) edge (c3) edge (c4) edge (c5) edge[bend left=12] (e);
  \path (b1) edge (c1) edge (c2);
  \path (b2) edge (c2) edge (c3);
  \path (b3) edge (c3) edge (c4);
  \path (b4) edge (c4') edge (c5);
  \path (b4') edge (c4);
  \path (b5) edge (c5);
  \node[rectangle,draw=none] at (0.25,-0.7) {$P_n$; $n\geq 2$};
\end{tikzpicture}
\end{minipage}\hskip 1cm
\begin{minipage}[b]{6cm}
\centering
\begin{tikzpicture}[xscale=1.25]
  \tikzstyle{every node}=[circle,minimum size=4pt,inner sep=0pt,draw]
  \tikzstyle{every label}=[rectangle,label distance=3pt,draw=none]
  \node (a) at (-0.5,0) {};
  \node (b1) at (-2,1) {};
  \node (b2) at (-1,1) {};
  \node (b3) at (0,1) {};
  \node[rectangle,draw=none,yshift=-0.27pt] at (0.75,1) {$\cdots$};
  \node (b4) at (1.5,1) {};
  \node (b5) at (2.5,1) {};
  \node (c1) at (-2,2) {};
  \node (c2) at (-1,2) {};
  \node (c3) at (0,2) {};
  \node[rectangle,draw=none,yshift=-0.27pt] at (0.75,2) {$\cdots$};
  \node (c4) at (1.5,2) {};
  \node (c5) at (2.5,2) {};
  \node[label=above:$1$] (d1) at (-2,3) {};
  \node[label=above:$2$] (d2) at (-1,3) {};
  \node[label=above:$3$] (d3) at (0,3) {};
  \node[rectangle,draw=none,yshift=-0.27pt] at (0.75,3) {$\cdots$};
  \node[label=above:$n$] (d4) at (1.5,3) {};
  \path (a) edge (b1) edge (b2) edge (b3) edge (b4);
  \path (c1) edge (b1) edge (d2) edge (d3) edge (d4);
  \path (c2) edge (d1) edge (b2) edge (d3) edge (d4);
  \path (c3) edge (d1) edge (d2) edge (b3) edge (d4);
  \path (c4) edge (d1) edge (d2) edge (d3) edge (b4);
  \path (c5) edge (b1) edge (b2) edge (b3) edge (b4) edge (b5);
  \node[rectangle,draw=none] at (0,-0.7) {$R_n$; $n\geq 3$};
\end{tikzpicture}
\end{minipage}
\end{center}
\caption{Irreducible posets with cut vertices}
\label{fig:cut-vertices}
\end{figure}
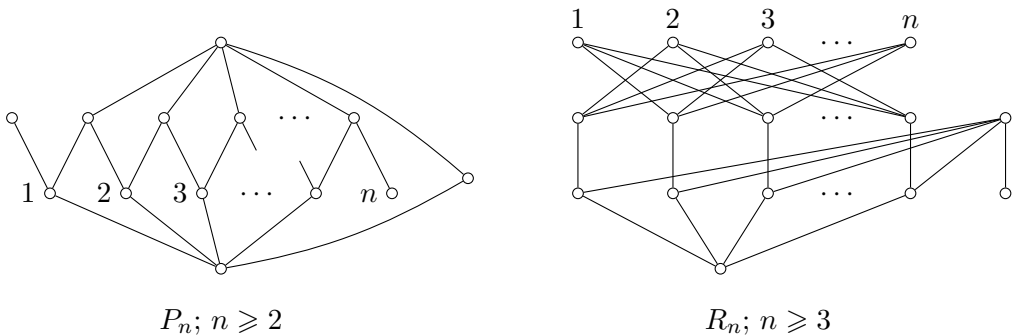

The following theorem is proved in \cite{TM77}.

\begin{theorem}
\label{thm:tree}
If\/ $P$ is a poset and the cover graph of\/ $P$ is a tree, then\/ $\dim(P)\leq 3$.
\end{theorem}

In Figure~\ref{fig:3dim-trees}, we show two posets whose cover graphs are trees.
These examples appear in \cite{TM77}, and we leave it as an exercise to verify that each of them has dimension $3$.
Accordingly, the inequality in Theorem~\ref{thm:tree} is best possible.
In the language of this paper, we note that when the cover graph of $P$ is a tree and $|P|\geq 2$, then every block of $P$ is a $2$-element chain and has dimension $1$.
Accordingly, in the case $d=1$, our main theorem reduces to a result which has been known for nearly $40$ years.
However, we emphasize that the proof we give in Section~\ref{sec:upper-bound} of the upper bound in Theorem~\ref{thm:main} is not inductive and works for all $d\geq 1$ simultaneously.
For this reason, it provides a new proof of Theorem~\ref{thm:tree} as a special case.

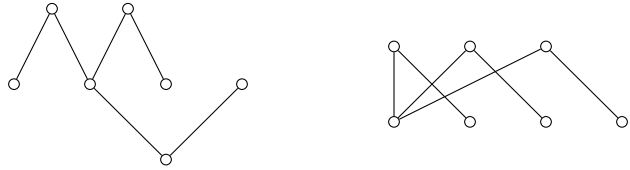
\begin{figure}
\begin{center}
\begin{tikzpicture}
  \tikzstyle{every node}=[circle,minimum size=4pt,inner sep=0pt,draw]
  \node (a) at (0,0) {};
  \node (b1) at (-2,1) {};
  \node (b2) at (-1,1) {};
  \node (b3) at (0,1) {};
  \node (b4) at (1,1) {};
  \node (c1) at (-1.5,2) {};
  \node (c2) at (-0.5,2) {};
  \path (a) edge (b2) edge (b4);
  \path (c1) edge (b1) edge (b2);
  \path (c2) edge (b2) edge (b3);
  \node (x1) at (3,0.5) {};
  \node (x2) at (4,0.5) {};
  \node (x3) at (5,0.5) {};
  \node (x4) at (6,0.5) {};
  \node (y1) at (3,1.5) {};
  \node (y2) at (4,1.5) {};
  \node (y3) at (5,1.5) {};
  \path (y1) edge (x1) edge (x2);
  \path (y2) edge (x1) edge (x3);
  \path (y3) edge (x1) edge (x4);
\end{tikzpicture}
\end{center}
\caption{$3$-Dimensional posets whose cover graphs are trees}
\label{fig:3dim-trees}
\end{figure}

A second paper in which trees and cut vertices are discussed is \cite{Tro78}, but the results of this paper are considerably stronger.
Here is a more recent result \cite{FTW15}, and only recently has the connection with blocks and cut vertices become clear.

\begin{theorem}
\label{thm:outerplanar}
If\/ $P$ is a poset and the cover graph of\/ $P$ is outerplanar, then\/ $\dim(P)\leq 4$.
Furthermore, this inequality is best possible.
\end{theorem}

In \cite{FTW15}, the construction shown in Figure~\ref{fig:diamonds} is given, and it is shown that when $n\geq 17$, the resulting poset has dimension $4$.
Note that the cover graph of this poset is outerplanar.
As a consequence, the inequality in Theorem~\ref{thm:outerplanar} is best possible.
Moreover, every block of $P_n$ is a $4$-element subposet having dimension $2$ (these subposets are called ``diamonds'').
We may then conclude that when $d=2$, the inequality in our main theorem is best possible.
However, the construction we present in Section~\ref{sec:best-possible} to show that our upper bound is best possible will again handle all values of $d$ with $d\geq 2$ at the same time, so it will not use this result either.

\begin{figure}
\begin{center}
\begin{tikzpicture}[scale=0.7]
  \tikzstyle{every node}=[circle,minimum size=4pt,inner sep=0pt,draw]
  \tikzstyle{every label}=[rectangle,label distance=3pt,draw=none]
  \node[label=left:$x$] (x) at (-2.5,0) {};
  \node[label=below:$a_1$] (a1) at (0,-1) {};
  \node[label=below:$a_2$] (a2) at (0,-2.25) {};
  \node[label=below:$a_3$] (a3) at (0,-3.5) {};
  \node[rectangle,draw=none] at (0,-4.5) {$\vdots$};
  \node[label=below:$a_n$] (a4) at (0,-5.5) {};
  \node[label=right:$b_1$] (b1) at (1,0) {};
  \node[label=right:$b_2$] (b2) at (2.25,0) {};
  \node[label=right:$b_3$] (b3) at (3.5,0) {};
  \node[rectangle,draw=none,yshift=-0.27pt] at (4.77,0) {$\cdots$};
  \node[label=right:$b_n$] (b4) at (5.5,0) {};
  \node[label=above:$c_1$] (c1) at (0,1) {};
  \node[label=above:$c_2$] (c2) at (0,2.25) {};
  \node[label=above:$c_3$] (c3) at (0,3.5) {};
  \node[rectangle,draw=none] at (0,4.8) {$\vdots$};
  \node[label=above:$c_n$] (c4) at (0,5.5) {};
  \path (x) edge (a1) edge (a2) edge (a3) edge (a4) edge (c1) edge (c2) edge (c3) edge (c4);
  \path (b1) edge (a1) edge (c1);
  \path (b2) edge (a2) edge (c2);
  \path (b3) edge (a3) edge (c3);
  \path (b4) edge (a4) edge (c4);
\end{tikzpicture}
\end{center}
\caption{A $4$-dimensional poset (when $n\geq 17$) with outerplanar cover graph}
\label{fig:diamonds}
\end{figure}
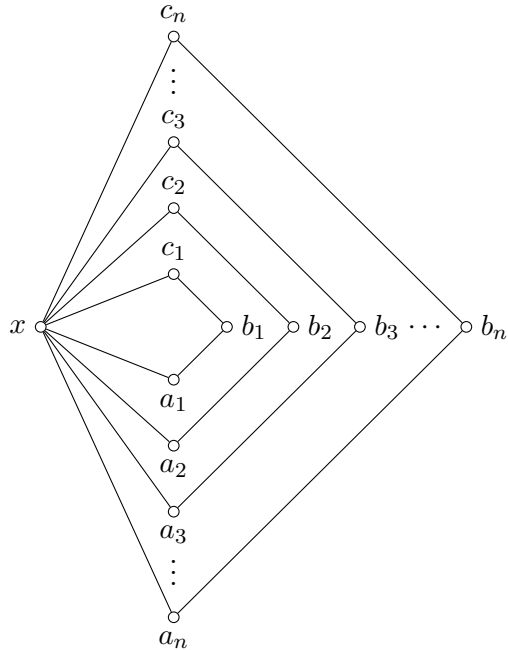

The results of this paper are part of a more comprehensive series of papers exploring connections between dimension of posets and graph-theoretic properties of their cover graphs.
Recent related papers include \cite{BKY16,JMM+16,JMT+17,JMW18,MW17,ST14,TW16,Wal17}.
However, many of these modern research themes have their roots in results, such as Theorem~\ref{thm:tree}, obtained in the 1970s or even earlier.

Here is one such example and again, Theorem~\ref{thm:tree} was the starting point.
The result is given in \cite{JMM+16}.

\begin{theorem}
\label{thm:tree-width}
For any positive integers\/ $t$ and\/ $h$, there is a least positive\/ $d=d(t,h)$ so that if\/ $P$ is a poset of height\/ $h$ and the cover graph of\/ $P$ has tree-width\/ $t$, then\/ $\dim(P)\leq d$.
\end{theorem}

As discussed in greater detail in \cite{JMM+16}, the function $d(t,h)$ must go to infinity with $h$ when $t\geq 3$, and in view of Theorem~\ref{thm:tree}, it is bounded for all $h$ when $t=1$.
These observations left open the question as to whether $d(2,h)$ is bounded or goes to infinity with $h$.
It is now known that $d(2,h)$ is bounded \cite{BKY16,JMT+17}.
However, in attacking this problem, the fact that one can restrict their attention to posets with $2$-connected cover graphs was a useful detail.

Also, the role of cut vertices in cover graphs surfaced in \cite{Wal17}, where the following result, which is considerably stronger than Theorem~\ref{thm:tree-width}, is proved.

\begin{theorem}
For any positive integers\/ $t$ and\/ $h$, there is a least positive integer\/ $d=d(n,h)$ so that if\/ $P$ is a poset of height at most\/ $h$ and the cover graph of\/ $P$ does not contain the complete graph\/ $K_n$ as a minor, then\/ $\dim(P)\leq d$.
\end{theorem}

The proof given in \cite{Wal17} uses the machinery of structural graph theory.
Subsequently, an alternative proof, using only elementary methods, was given in \cite{MW17}, and an extension to classes of graphs with bounded expansion was given in \cite{JMW18}.

\section{Dimension theory essentials}
\label{sec:dimension}

Let $P$ be a poset with ground set $X$.
Then, let $\Inc(P)$ denote the set of all ordered pairs $(x,y)\in X\times X$ where $x$ is incomparable to $y$ in $P$.
The binary relation $\Inc(P)$ is of course symmetric, and it is empty when $P$ is a total order---in this case, $\dim(P)=1$.

A subset $R\subseteq\Inc(P)$ is \emph{reversible} when there is a linear extension $L$ of $P$ so that $x>y$ in $L$ for all $(x,y)\in R$.
When $\Inc(P)\neq\emptyset$, the dimension of $P$ is then the least positive integer $d$ for which there is a covering
\[\Inc(P)=R_1\cup R_2\cup\cdots\cup R_d\]
such that $R_j$ is reversible for each $j=1,2,\ldots,d$.

In the proof of the upper bound in our main theorem, we will apply these observations to show that a poset $P$ has dimension at most $d+2$ by first constructing a family $\cgF=\{L_1,L_2,\ldots,L_d\}$ of linear extensions of $P$ and then setting $R=\{(x,y)\in P\colon x>y$ in $L_j$ for all $j=1,2,\ldots,d\}$.
If $R=\emptyset$, then $\dim(P)\leq d$, and when $R\neq\emptyset$, we will find a covering $R=R_{d+1}\cup R_{d+2}$, where both $R_{d+1}$ and $R_{d+2}$ are reversible.
If $L_{d+1}$ and $L_{d+2}$ are linear extensions of $P$, so that $x>y$ in $L_j$ whenever $(x,y)\in R_j$ for each $j=d+1,d+2$, then $\cgR=\{L_1,L_2,\ldots,L_d,L_{d+1},L_{d+2}\}$ is a realizer of $P$, which shows $\dim(P)\leq d+2$.

An indexed subset $\{(x_i,y_i)\colon 1\leq i\leq k\}\subseteq\Inc(P)$ is called an \emph{alternating cycle} of length $n$ when $x_i\leq y_{i+1}$ in $P$, for all $i=1,2,\ldots,k$ (here, subscripts are interpreted cyclically so that $x_n\leq y_1$ in $P$).
In \cite{TM77}, the following elementary result is proved.

\begin{lemma}
Let\/ $P$ be a poset and let\/ $R\subseteq\Inc(P)$.
Then\/ $R$ is reversible if and only if\/ $R$ does not contain an alternating cycle.
\end{lemma}

The following construction was given in \cite{DM41}, where the concept of dimension was introduced.
For an integer $d\geq 2$, let $S_d$ be the following height $2$ poset: $S_d$ has $d$ minimal elements $\{a_1,a_2,\ldots,a_d\}$ and $d$ maximal elements $\{b_1,b_2,\ldots,b_d\}$; the partial ordering on $S_d$ is defined by setting $a_i<b_j$ in $S_d$ if and only if $i\neq j$.
The poset $S_d$ is called the \emph{standard example} (of dimension $d$).

In dimension theory, standard examples play a role which in many ways parallels the role of complete graphs in the study of chromatic number, and we refer the reader to \cite{BHPT16} for additional details on extremal problems for which results for graphs and results for posets have a similar flavor.
In this paper, we will only need the following basic information about standard examples, which was noted in \cite{DM41}.

\begin{proposition}
For every\/ $d\geq 2$, we have\/ $\dim(S_d)=d$.
In fact, if\/ $S_d$ is a standard example, and\/ $a_i>b_i$ in a linear extension\/ $L$ of\/ $S_d$ then\/ $a_j<b_j$ in\/ $L$ whenever\/ $i\neq j$.
\end{proposition}

\section{Proof of the upper bound}
\label{sec:upper-bound}

Before launching into the main body of the proof, we pause to present an important proposition, which will be very useful in the argument to follow.
When $M$ is a linear extension of a poset $P$ and $w\in P$, we will write $M=[A<w<B]$ when the elements of $P$ can be labeled so that $M=[u_1<u_2<\cdots<u_m]$, $A=[u_1<u_2<\cdots<u_{k-1}]$, $w=u_k$, and $B=[u_{k+1}<u_{k+2}<\cdots<u_m]$.
The generalization of this notation to an expression such as $M=[A<C<w<D<B]$ should be clear.
Given this notation, the following is nearly self-evident, but it is stated for emphasis.

\begin{proposition}
\label{pro:merge}
Let\/ $P$ be a poset, and let\/ $w$ be a cut vertex in\/ $P$.
Let\/ $P'$ and\/ $P''$ be subposets of\/ $P$ such that\/ $P'\cap P''=\{w\}$ and the vertex\/ $w$ separates\/ $P'$ and\/ $P''$ in the cover graph of\/ $P$.
If\/ $M'=[A<w<B]$ and\/ $M''=[C<w<D]$ are linear extensions of\/ $P'$ and\/ $P''$, respectively, then\/ $M=[A<C<w<D<B]$ is a linear extension of the subposet of $P$ induced on\/ $P'\cup P''$.
Furthermore, the restriction of\/ $M$ to\/ $P'$ is\/ $M'$ and the restriction of\/ $M$ to\/ $P''$ is\/ $M''$.
\end{proposition}

The rule $M=[A<C<w<D<B]$ will be called the \emph{merge rule}.
When it is applied, we will consider $M'=[A<w<B]$ as a linear extension of an ``old'' subposet $P'$ which shares a cut vertex $w$ with a ``new'' subposet $P''$ for which $M''=[C<w<D]$ is a linear extension.
We apply the merge rule to form a linear extension $M=[A<C<w<D<B]$ of the union $P'\cup P''$ and note that $M$ forces old points in $A\cup B$ to the outside while concentrating new points from $C\cup D$ close to $w$.

Now on to the proof.
We fix a positive integer $d\geq 1$ and let $P$ be a poset for which $\dim(B)\leq d$ for every block $B$ of $P$.
The remainder of the proof is directed towards proving that $\dim(P)\leq d+2$.
Let $G$ be the cover graph of $P$.
Since $d+2\geq 3$, we may assume that $G$ is connected.

Let $\cgB$ be the family of blocks in $P$, and let $t=|\cgB|$.
Then, let $\cgB=\{B_1,B_2,\ldots,B_t\}$ be any labeling of the blocks of $P$ such that for every $i=2,3,\ldots,t$, one of the vertices of $B_i$ belongs to some of the blocks $B_1,B_2,\ldots,B_{i-1}$.
Such a vertex of $B_i$ is unique and is a cut vertex of $P$---we call it the \emph{root} of $B_i$ and denote it by $\rho(B_i)$.

For every block $B_i\in\cgB$ and every element $u\in B_i$, we define the \emph{tail} of $u$ relative to $B_i$, denoted by $T(u,B_i)$, to be the subposet of $P$ consisting of all elements $v\in\{u\}\cup B_{i+1}\cup B_{i+2}\cup\cdots\cup B_t$ for which every path from $v$ to any vertex in $B_i$ passes through $u$.
Note that $T(u,B_i)=\{u\}$ if $u$ is not a cut vertex.
Also, if $u\in B_i$, $v\in B_{i'}$, and $(u,i)\neq(v,i')$, then either $T(u,B_i)\cap T(v,B_{i'})=\emptyset$ or one of $T(u,B_i)$ and $T(v,B_{i'})$ is a proper subset of the other.

For every block $B_i\in\cgB$, using the fact that $\dim(B_i)\leq d$, we may choose a realizer $\cgR_i=\{L_j(B_i)\colon 1\leq j\leq d\}$ of size $d$ for $B_i$.
For each $i=1,2,\ldots,t$, set $P_i=B_1\cup B_2\cup\cdots\cup B_i$.
Note that when $2\leq i\leq t$, we have $\rho(B_i)\in P_{i-1}$.

Fix an integer $j$ with $1\leq j\leq d$ and set $M_j(1)=L_j(B_1)$.
Then, repeat the following for $i=2,3,\ldots,t$.
Suppose that we have a linear extension $M_j(i-1)$ of $P_{i-1}$.
Let $w=\rho(B_i)$.
Since $w\in P_{i-1}$, we can write $M_j(i-1)=[A<w<B]$.
If $L_j(B_i)=[C<w<D]$, we then use the merge rule to set $M_j(i)=[A<C<w<D<B]$.
When the procedure halts, take $L_j=M_j(t)$.
This construction is performed for all $j=1,2,\ldots,d$ to determine a family $\cgF=\{L_1,L_2,\ldots,L_d\}$
of linear extensions of $P$.
The family $\cgF$ is a realizer for a poset $P^*$ which is an extension of $P$.
As outlined in the preceding section, we set $R=\{(x,y)\in\Inc(P)\colon x<y$ in $L_j$ for every $j=1,2,\ldots,d\}$.
We will show that there is a covering $R=R_{d+1}\cup R_{d+2}$ of $R$ by two reversible sets.
This is enough to prove $\dim(P)\leq d+2$.

Repeated application of Proposition~\ref{pro:merge} immediately yields the following.

\begin{block-prop}
For each $j=1,2,\ldots,d$ and each block $B_i\in\cgB$, the restriction of $L_j$ to $B_i$ is $L_j(B_i)$.
\end{block-prop}

When $L$ is a linear order on a set $X$ and $S\subseteq X$, we say $S$ is an \emph{interval} in $L$ if $y\in S$ whenever $x,z\in S$ and $x<y<z$ in $L$.
The next property follows easily from the observation that the merge rule concentrates new points close around the cut vertex $w$ while pushing old points to the outside.

\begin{interval-prop}
For every $j=1,2,\ldots,d$, and every pair $(u,i)$ with $u\in B_i$, the tail $T(u,B_i)$ of $u$ relative to $B_i$ is an interval in $L_j$.
\end{interval-prop}

Let $(x,y)\in R$.
Then, let $i$ be the least positive integer for which every path from $x$ to $y$ in the cover graph of $P$ contains at least
two elements of the block $B_i$.
We then define elements $u,v\in B_i$ by the following rules:
\begin{enumerate}
\item $u$ is the unique first common element of $B_i$ with every path from $x$ to $y$;
\item $v$ is the unique last common element of $B_i$ with every path from $x$ to $y$.
\end{enumerate}
Note that $u\neq v$, $u=x$ when $x\in B_i$, and $v=y$ when $y\in B_i$.

\begin{claim}
\label{cla:1}
The following two statements hold:
\begin{enumerate}
\item $x\in T(u,B_i)$, $y\notin T(u,B_i)$, $y\in T(v,B_i)$, and\/ $x\notin T(v,B_i)$;
\item\label{item:cla:1.2} $u<v$ in\/ $P$.
\end{enumerate}
\end{claim}

\begin{proof}
The first statement is an immediate consequence of the definition of tails.
For the proof of the second statement, suppose to the contrary that $u\not<v$ in $P$.
Since $u,v\in B_i$ and $u\neq v$, there is some $j$ with $1\leq j\leq d$ so that $u>v$ in $L_j(B_i)$.
Therefore, $u>v$ in $L_j$.
Since, $T(u,B_i)$ and $T(v,B_i)$ are disjoint intervals in $L_j$, we conclude that $x>y$ in $L_j$.
This contradiction shows $u<v$ in $P$, as claimed.
\end{proof}

\begin{claim}
\label{cla:2}
At least one of the following two statements holds:
\begin{enumerate}
\item for all\/ $y'$ with\/ $y'\geq x$ in\/ $P$, we have\/ $y'\in T(u,B_i)$ and\/ $y'<y$ in\/ $P^*$;
\item for all\/ $x'$ with\/ $x'\leq y$ in\/ $P$, we have\/ $x'\in T(v,B_i)$ and\/ $x<x'$ in\/ $P^*$.
\end{enumerate}
\end{claim}

\begin{proof}
Suppose to the contrary that neither of the two statements holds.
Since $T(u,B_i)$ is an interval in $L_j$, $x\in T(u,B_i)$, $y\notin T(u,B_i)$, and $x<y$ in $L_j$ for each $j=1,2,\ldots,d$, there must exist some $y'$ with $y'>x$ in $P$ and $y'\notin T(u,B_i)$.
Then a path in $G$ from $x$ to $y'$ witnessing the inequality $x<y'$ in $P$ must include the point $u$.
In particular, this implies that $x\leq u$ in $P$.
Similarly, we have $v\leq y$ in $P$.
This and Claim \ref{cla:1}~\ref{item:cla:1.2} yield $x\leq u<v\leq y$ in $P$, which is a contradiction.
\end{proof}

Now, it is clear how to define the covering $R=R_{d+1}\cup R_{d+2}$.
We assign $(x,y)$ to $R_{d+1}$ when the first statement in Claim~\ref{cla:2} applies, and we assign it to $R_{d+2}$ when the second statement applies.
We show that $R_{d+1}$ is reversible.
The argument for $R_{d+2}$ is analogous.
Suppose to the contrary that $R_{d+1}$ is not reversible.
Then there is an integer $n\geq 2$ and an alternating cycle $\{(x_i,y_i)\colon 1\leq i\leq n\}$ contained in $R_{d+1}$.
Then $x_i\leq y_{i+1}$ for each $i=1,2,\ldots,n$.
However, since $(x_i,y_i)\in R_{d+1}$, we know that $y_{i+1}<y_i$ in $P^*$ for every $i=1,2,\ldots,n$.
This is impossible, because $P^*$ is a partial order.
The proof of the upper bound in Theorem~\ref{thm:main} is now complete.

\section{Proof that the upper bound is best possible}
\label{sec:best-possible}

As we noted previously, the examples shown in Figure~\ref{fig:3dim-trees} show that our upper bound is best possible when $d=1$.
So in this section, we will fix an integer $d\geq 2$ and show that there is a poset $P$ so that $\dim(P)=d+2$ while $\dim(B)\leq d$ for every block $B$ of $P$.

For a positive integer $n$, we let $\bfn$ denote the $n$-element chain $\{0<1<\cdots<n-1\}$.
Also, we let $\bfn^d$ denote the Cartesian product of $d$ copies of $\bfn$, that is, the elements of $\bfn^d$ are $d$-tuples of the form $u=(u_1,u_2,\ldots,u_d)$ where each coordinate $u_i$ is an integer with $0\leq u_i<n$.
The partial order on $\bfn^d$ is defined by setting $u=(u_1,u_2,\ldots,u_d)\leq(v_1,v_2,\ldots,v_d)=v$ in $\bfn^d$ if and only if $u_i\leq v_i$ in $\bfn$ for all $i=1,2,\ldots,d$.
As is well known, $\dim(\bfn^d)=d$ for all $n\geq 2$.

For each $n\geq 2$, we then construct a poset $P=P(n)$ as follows.
We start with a base poset $W$ which is a copy of $\bfn^d$.
The base poset $W$ will be a block in $P$, and $W$ will also be the set of cut vertices in $P$.
All other blocks in $P$ will be ``diamonds'', that is, copies of the $2$-dimensional poset on four points discussed in conjunction with Figure~\ref{fig:diamonds}.
Namely, for each element $w\in W$, we attach a $3$-element chain $x_w<y_w<z_w$ so that $x_w<w<z_w$ while $w$ is incomparable to $y_w$.
In this way, the $4$-element subposet $\{w,x_w,y_w,z_w\}$ is a diamond.

We will now prove the following claim.

\begin{claim}
If\/ $n$ is sufficiently large, then\/ $\dim(P)\geq d+2$.
\end{claim}

\begin{proof}
We must first gather some necessary tools from Ramsey theory.
In particular, we need a special case of a result which has become known as the Product Ramsey Theorem and appears in the classic text \cite{GRS-book} as Theorem~5 on page 113.
However, we will use slightly different notation in discussing this result.

When $T_1,T_2,\ldots,T_d$ are $k$-element subsets of $X_1,X_2,\ldots,X_d$, respectively, we refer to the product $g=T_1\times T_2\times\cdots\times T_d$ as a \emph{$\bfk^d$-grid} in $X_1\times X_2\times\cdots\times X_d$.
Here is a formal statement of the version of the Product Ramsey Theorem we require for our proof.

\begin{theorem}
\label{thm:prod-ramsey}
For every\/ $4$-tuple\/ $(r,d,k,m)$ of positive integers with\/ $m\geq k$, there is an integer\/ $n_0\geq k$ such that if\/ $|X_i|\geq n_0$ for every\/ $i=1,2,\ldots,d$, then whenever we have a coloring\/ $\phi$ which assigns to each\/ $\bfk^t$-grid\/ $g$ in\/ $X_1\times X_2\times\cdots\times X_d$ a color\/ $\phi(g)$ from a set\/ $R$ of\/ $r$ colors, then there is a color\/ $\alpha\in R$ and there are\/ $m$-element subsets\/ $H_1,H_2,\ldots,H_d$ of\/ $X_1,X_2,\ldots,X_d$, respectively, such that\/ $\phi(g)=\alpha$ for every\/ $\bfk^t$-grid\/ $g$ in\/ $H_1\times H_2\times\cdots\times H_d$.
\end{theorem}

We will apply this theorem with $k=2$, and since $k$ and $d$ are now both fixed, we will just refer to a $\bfk^d$-grid as a grid.
When $g=T_1\times T_2\times\cdots\times T_d$ is a grid, we consider the elements $w\in W$ with $w_j\in T_j$ for each $j=1,2,\ldots,d$.
Clearly, there are $2^d$ such points.
Counting the diamonds attached to these points, there are $4\cdot 2^d$ points in $P$ associated with the grid $g$.
But we want to focus on $4d+8$ of them.

First, we consider an antichain $A=\{a_1,a_2,\ldots,a_d\}$ defined as follows: for each $i,j=1,2,\ldots,d$, coordinate $j$ of $a_i$ is $\max(T_j)$ when $i=j$ and $\min(T_j)$ when $i\neq j$.
Dually, the antichain $B=\{b_1,b_2,\ldots,b_d\}$ is defined as follows: for each $i,j=1,2,\ldots,d$, coordinate $j$ of $b_i$ is $\min(T_j)$ when $i=j$ and $\max(T_j)$ when $i\neq j$.
We then note that when $1\leq i,j\leq d$, we have $a_i<b_j$ in $P$ if and only if $i\neq j$.
As a consequence, the subposet of $P$ determined by $A\cup B$ is the standard example $S_d$ discussed previously.
Note further that the points in the two antichains $\{x_{a_j}\colon 1\leq j\leq d\}$ and $\{z_{b_j}\colon 1\leq j\leq d\}$ also form a copy of $S_d$.
Furthermore, if $x_{a_j}>z_{b_j}$ in some linear extension $L$ of $P$, then $a_j>b_j$ in $L$.

Now, we consider two special points $c$ and $d$ in $W$ associated with the grid $g$ together with the points in the diamonds attached at $c$ and $d$.
First, we take $c$ with $c_j=\min(T_j)$ for each $j=1,2,\ldots,d$, and then we take $d$ with $d_j=\max(T_j)$ for each $j=1,2,\ldots,d$.
We note that $x_c<c<d<z_d$ in $P$.
However, we also note that both $(x_c,y_d)$ and $(y_c,z_d)$ are in $\Inc(P)$.

Now, suppose that $\dim(P)\leq d+1$ and that $\cgR=\{L_1,L_2,\ldots,L_{d+1}\}$ is a realizer of $P$.
We will argue to a contradiction provided $n$ is sufficiently large.
To accomplish this, we define a coloring $\phi$ of the grids in $\bfn^d$ using $(d+1)^{d+2}$ colors.
Let $g$ be a grid and consider the $2d+4$ points discussed above.
The color $\phi(g)$ will be a vector $(\alpha_1,\alpha_2,\ldots,\alpha_{d+2})$ of length $d+2$ defined as follows:
\begin{enumerate}
\item for $j=1,2,\ldots,d$, $\alpha_j$ is the least index $\alpha$ for which $x_{a_j}>z_{b_j}$ in $L_\alpha$;
\item $\alpha_{d+1}$ is the least index $\alpha$ for which $x_c>y_d$ in $L_\alpha$;
\item $\alpha_{d+2}$ is the least index $\alpha$ for which $y_c>z_d$ in $L_\alpha$.
\end{enumerate}

We apply Theorem~\ref{thm:prod-ramsey} with $r=(d+1)^{d+2}$ and $m=3$.
It follows that there is some fixed color $(\beta_1,\beta_2,\ldots,\beta_{d+2})$ such that for each $j=1,2,\ldots,d$, there is a $3$-element subset $H_j=\{u_{1,j}<u_{2,j}<u_{3,j}\}\subseteq\{0,1,\ldots,n-1\}$ so that $\phi(g)=(\beta_1,\beta_2,\ldots,\beta_{d+2})$ for all grids $g$ in $H_1\times H_2\times\cdots\times H_d$.

First, we claim that $\beta_{d+1}\neq\beta_{d+2}$.
To see this, suppose that $\beta=\beta_{d+1}=\beta_{d+2}$.
Then, let $c,c',d,d'$ be vectors with $c_j=u_{1,j}$, $d_j=c'_j=u_{2,j}$ and $d'_j=u_{3,j}$ for each $j=1,2,\ldots,d$.
Then
\[c>x_c>y_d=y_{c'}>z_{d'}>d'>c\quad\text{in $L_\beta$}.\]
Clearly, this is impossible.

In view of our earlier remarks concerning standard examples, we may relabel the linear extensions in the realizer so that $\beta_j=j$ for each $j=1,2,\ldots,d$.
Given the fact that $\beta_{d+1}\neq\beta_{d+2}$, (at least) one of $\beta_{d+1}$ and $\beta_{d+2}$ is in $\{1,2,\ldots,d\}$.
We complete the argument assuming that $1\leq j=\beta_{d+1}\leq d$, noting that the other case is analogous.

Now, let $c,d,e$ be points in $W$ defined as follows.
First, set $c_j=u_{1,j}$, $e_j=u_{2,j}$ and $d_j=u_{3,j}$.
Then, for each $i=1,2,\ldots,d$ with $i\neq j$, set $c_i=u_{1,i}$, $d_i=u_{2,i}$ and $e_i=u_{3,i}$.
Note that $c<d$ and $c<e$ in $P$.
It follows then that
\[c>x_c>y_d>x_d>z_e>e>c\quad\text{in $L_j$},\]
where the inequality $x_d>z_e$ follows from the fact that $\phi(g)=(\beta_1,\beta_2,\ldots,\beta_{d+2})$ for the grid $g=\{u_{2,1},u_{3,1}\}\times\{u_{2,2},u_{3,2}\}\times\cdots\times\{u_{2,d},u_{3,d}\}$.
The contradiction completes the proof that the upper bound in our main theorem is best possible.
\end{proof}

\section{Closing comments}
\label{sec:closing}

There are two other instances where the Product Ramsey Theorem has been applied to combinatorial problems for posets, although, to be completely accurate, the first only uses it implicitly.
The following two inequalities are proved in \cite{Tro75}.

\begin{theorem}
\label{thm:max(P)}
Let\/ $P$ be a poset which is not an antichain, and let\/ $w$ be the width of the subposet\/ $P-\Max(P)$, where\/ $\Max(P)$ is the set of maximal elements of\/ $P$.
Then\/ $\dim(P)\leq w+1$.
\end{theorem}

\begin{theorem}
Let\/ $A$ be an antichain in a poset\/ $P$.
If\/ $P-A\neq\emptyset$ and\/ $w$ is the width of\/ $P-A$, then\/ $\dim(P)\leq 2w+1$.
\end{theorem}

Both results admit quite simple proofs, and the only real challenge is to show that they are best possible.
An explicit construction is given in \cite{Tro75} for a family of posets showing that Theorem~\ref{thm:max(P)} is tight, but the construction for the second is far more complicated and deferred to a separate paper \cite{Tro74}.
Readers who are familiar with the details of this construction will recognize that it is an implicit application of the Product Ramsey Theorem.

The second application appears in \cite{FFT99} where it is shown that there is a \emph{finite} $3$-dimensional poset which cannot be represented as a family of spheres in Euclidean space---of any dimension---ordered by inclusion.
While it may in fact be the case that such posets exist with only a few hundred points, the proof produces an example which is extraordinarily large.
This results from the fact that a further strengthening of the Product Ramsey Theorem to a lexicographic version (see \cite{FG93}) is required.

We consider it a major challenge to construct examples of modest size for each of these three problems to replace the enormous posets resulting from the application of the Product Ramsey Theorem.

\end{document}